\pdfoutput=1
\RequirePackage{ifpdf}
\ifpdf % We are running pdfTeX in pdf mode
\documentclass[pdftex]{sigma}
\else
\documentclass{sigma}
\fi

\DeclareMathOperator{\Id}{Id}
\DeclareMathOperator{\id}{id}
\def\vectorfields#1{{\mathcal X}(#1)}
\def\fpd#1#2{{\displaystyle\frac{\partial #1}{\partial #2}}}
\newcommand{\la}{\mathfrak{g}}

\newcommand{\lam}{\mathfrak{m}}
\newcommand{\lan}{\mathfrak{n}}
\newcommand{\lak}{\mathfrak{k}}
\newcommand{\so}{\mathfrak{so}}
\newcommand{\gl}{\mathfrak{gl}}
\newcommand{\sli}{\mathfrak{sl}}
\newcommand{\R}{\mathbb{R}}
\newcommand{\Sph}{\mathbb{S}}

\begin{document}

%\allowdisplaybreaks

\newcommand{\arXivNumber}{1606.07649}

\renewcommand{\PaperNumber}{115}

\FirstPageHeading

\ShortArticleName{Un-Reduction of Systems of Second-Order Ordinary Dif\/ferential Equations}

\ArticleName{Un-Reduction of Systems of Second-Order\\ Ordinary Dif\/ferential Equations}

\Author{Eduardo GARC\'IA-TORA\~NO ANDR\'ES~$^\dag$ and Tom MESTDAG~$^\ddag$}

\AuthorNameForHeading{E.~Garc\'{\i}a-Tora\~no Andr\'es and T.~Mestdag}

\Address{$^\dag$~Departamento de Matem\'atica, Universidad Nacional del Sur, CONICET, \\
\hphantom{$^\dag$}~Av.\ Alem 1253, 8000 Bah\'ia Blanca, Argentina}
\EmailD{\href{mailto:egtoranoandres@gmail.com}{egtoranoandres@gmail.com}}

\Address{$^\ddag$~Department of Mathematics and Computer Science, University of Antwerp,\\
\hphantom{$^\ddag$}~Middelheimlaan 1, B--2020 Antwerpen, Belgium}
\EmailD{\href{mailto:mestdagtom@gmail.com}{mestdagtom@gmail.com}}

\ArticleDates{Received August 12, 2016, in f\/inal form November 29, 2016; Published online December 07, 2016}

\Abstract{In this paper we consider an alternative approach to ``un-reduction''. This is the process where one associates to a Lagrangian system on a manifold a dynamical system on a principal bundle over that manifold, in such a way that solutions project. We show that, when written in terms of second-order ordinary dif\/ferential equations (SODEs), one may associate to the f\/irst system a (what we have called) ``primary un-reduced SODE'', and we explain how all other un-reduced SODEs relate to it. We give examples that show that the considered procedure exceeds the realm of Lagrangian systems and that relate our results to those in the literature.}


\Classification{34A26; 37J15; 70H33; 70G65}

\section{Introduction}

One of the much-discussed aspects of Lagrangian systems with a symmetry group is their reduction to the so-called Lagrange--Poincar\'e equations \cite{CMR} (but see also \cite{review} for an approach using Lie algebroids, or \cite{MC} for an approach that is relevant for this paper). The idea is that one may cancel out the symmetry group, and, once a principal connection has been invoked, arrive at two sets of coupled equations on a quotient manifold, the so-called horizontal and vertical equations. The horizontal equation looks again a bit like a Lagrangian equation (for a reduced-type Lagrangian function), but it has extra non-conservative force terms and, equally important, it is in general not decoupled from the vertical equation.

In recent years, there has been some interest in so-called un-reduction \cite{unred2,unred3,unreduction, CH}. The motivation behind this paper is related to the un-reduction theorem (Theorem~5.1) of the paper~\cite{unreduction}. Given a Lagrangian~$\ell$ on a quotient manifold $\bar M=M/G$ and a principal connection~$\omega$ on~$\pi\colon M \to M/G$, the theorem tells one how to associate a ``Lagrangian system with extra non-conservative forces'' on $M$ which has the property that its solutions project to those of the Euler--Lagrange equations of~$\ell$. The methodology and reasoning of~\cite{unreduction} is almost entirely based on Lagrange--Poincar\'e reduction, and on the fact that one may choose the external forces in such a way that the inconvenient curvature terms which appear in the horizontal Lagrange--Poincar\'e equation vanish. In our opinion, it may be more advantageous to think of this problem outside of the Lagrange--Poincar\'e framework. The Lagrange--Poincar\'e equations give essentially a dynamical system on the manifold $(TM)/G$ (a Lie algebroid), while the Euler--Lagrange equations of $\ell$ are def\/ined on $T(M/G)$ (a tangent bundle). One of the purposes of this paper is to show that for a comparison of dynamics one may remain in the category of tangent bundles.

The Euler--Lagrange equations of $\ell$ are but a particular example of a SODE, a system of second-order ordinary dif\/ferential equations. The un-reduced equations one f\/inds in \cite{unreduction} are no longer Euler--Lagrange equations, since the presence of extra non-conservative forces breaks the variational character of these equations, but they are still a SODE. We will show that the un-reduction process is very natural in the context of SODEs, and that one may identify a~(what we call) ``primary un-reduced SODE'', to which all other un-reduced SODEs easily relate. There is, in our setting, no need to invoke Lagrange--Poincar\'e equations or external forces. The issues related to what is called ``coupling distortion'' and ``curvature distortion'' in \cite{unreduction} are, in our opinion, side-ef\/fects from relying on a Lagrange--Poincar\'e-based approach to un-reduction (as opposed to a SODE-based approach).

In Section~\ref{sectionlifted} we identify the two lifted principal connections that lie at the basis of our un-reduction. After the def\/inition of the primary un-reduced SODE in Section~\ref{sectionsode} we state a~proposition about all other SODEs whose base integral curves project on those of the given SODE. The canonical connection of a Lie group gives, in Section~\ref{sectioncanonical}, a natural environment where all the introduced concepts can be clarif\/ied. Since it is not always variational (i.e., its geodesics are not always solutions of some Euler--Lagrange equations) it shows that our discussion is a~meaningful generalization of the one in~\cite{unreduction}. In Section~\ref{sectionLagrangian} we specify to the case of a Lagrangian SODE, and we discuss the example of Wong's equations and the ef\/fect of curvature distortion in Section~\ref{sectionvertical}. In Section~\ref{sectionsecond} we say a few words about a second un-reduction procedure. We end the paper with some possible lines of future research.

The advantage of un-reduction remains unaltered in our setting, as it is explained in \cite{unreduction}: If one knows that solutions of a second-order system on $\bar M$ are but the projection of those of a~system on $M$, one may compute these solutions by making use of any coordinates on $M$, not necessarily those adapted to the bundle structure of~$\pi$. This may be benef\/icial when, e.g., the equations on $M$ are more convenient to deal with numerically, which is precisely one of the main motivations to study un-reduction. We refer the interested reader to~\cite{unreduction} and references therein for a more detailed discussion.

\section{Preliminaries}

Consider a manifold $M$ and its tangent bundle $\tau_M\colon TM\to M$. The vector f\/ields $X^C$ and $X^V$ on $TM$ stand, respectively, for the complete lift and the vertical lift of a vector f\/ield $X$ on $M$. In natural coordinates $\big(x^A,{\dot x}^A\big)$ on $TM$ and for $X=X^A \partial/\partial x^A$ we get
\begin{gather*}
X^C = X^A \fpd{}{{x}^A} + {\dot x}^B\fpd{X^A}{x^B} \fpd{}{{\dot x}^A} \qquad \text{and} \qquad X^V= X^A \fpd{}{{\dot x}^A}.
\end{gather*}
The Lie brackets of these vector f\/ields are $\big[X^C,Y^C\big] = [X,Y]^C$, $\big[X^C,Y^V\big] = [X,Y]^V$ and $\big[X^V,Y^V\big]=0$.

The notion of complete and vertical lifts also extends to functions and $(1,1)$ tensor f\/ields, as follows. For a function $f$ on $M$, its lifts are the functions $f^V =f\circ \tau_M$ and $f^C = {\dot x}^A \big(\partial f/\partial x^A\big)$ on~$TM$. Let $A$ be a $(1,1)$ tensor f\/ield on~$M$. We may lift it to two $(1,1)$ tensor f\/ields~$A^C$ and~$A^V$ on~$TM$, as follows
\begin{gather*}
A^C\big(X^V\big)=(A(X))^V,\!\!\!\!\qquad A^C\big(X^C\big)=(A(X))^C,\!\!\!\!\qquad A^V\big(X^V\big)=0,\!\!\!\!\qquad A^V\big(X^C\big)=(A(X))^V.
\end{gather*}
More details on this can be found in \cite{YI}, such as the following immediate properties:
\begin{gather*}
{\mathcal L}_{X^V}A^C = ({\mathcal L}_XA)^V, \qquad {\mathcal L}_{X^C}A^C = ({\mathcal L}_XA)^C, \qquad {\mathcal L}_{X^V}A^V = 0, \qquad {\mathcal L}_{X^C}A^V = ({\mathcal L}_XA)^V.
\end{gather*}

\begin{definition}
A vector f\/ield $\Gamma$ on $TM$ is a second-order dif\/ferential equations f\/ield (SODE in short) on $M$ if all its integral curves $\gamma\colon I \to TM$ are lifted curves, that is of the type $\gamma = \dot c$, for base integral curves $c\colon I \to M$.
\end{definition}
As such, $\Gamma$ takes the form
\begin{gather*}
\Gamma = {\dot x}^A \fpd{}{x^A}+ f^A(x,{\dot x}) \fpd{}{{\dot x}^A}.
\end{gather*}
As a vector f\/ield on $TM$, a SODE is characterized by the property that $T\tau_M \circ\Gamma = \id_M$. A~SODE~$\Gamma$ on~$M$ always exists, since a Riemannian metric on a (paracompact) manifold always exists, and we may consider its geodesic spray.

In what follows we will often express a SODE in terms of a frame of vector f\/ields $Z_A = Z_A^B \partial/\partial x^B$, which are not necessarily coordinate vector f\/ields. In that context, we say that the quasi-velocities $\big(v^A\big)$ of a vector $v_m\in T_mM$ are the components of~$v_m$ with respect to that basis, i.e., $v_m=v^A Z_A(m)$. Their relation to the standard f\/ibre coordinates ${\dot x}^A$ is ${\dot x}^A = Z^A_B(m) v^B$.

In terms of the frame $\{Z_A\}$, the SODE $\Gamma$ takes the form
 \begin{gather*}
 \Gamma = v^A Z_A^C + F^A Z_A^V,
 \end{gather*}
meaning that two SODEs $\Gamma_1$ and $\Gamma_2$ only dif\/fer in their coef\/f\/icients $F_1^A$ and $F_2^A$. Let $c(t) = \big(x^A(t)\big)$ be a base integral curve of $\Gamma$. The f\/ibre coordinates of integral curves $\dot c(t)$ of $\Gamma$ in $TM$ may also be expressed in quasi-velocities as ${\dot c}(t) = v^A(t) Z_A(c(t))$. The functions $\big(x^A(t),v^A(t)\big)$ are then solutions of the equations
 \begin{gather*}
{\dot x}^A(t) = Z^A_B(x(t))v^B(t), \qquad {\dot v}^A = F^A (x(t),v(t)).
 \end{gather*}

Assume that $M$ comes equipped with a free and proper (left) action $\Phi\colon G \times M \to M $ of a Lie group $G$, such that $\pi\colon M \to \bar M=M/G$ is a principal f\/ibre bundle. For each element $\xi$ in the Lie algebra $\la$ we may obtain a fundamental vector f\/ield $\xi_M$ on $M$, def\/ined by $\xi_M(m) = T\Phi_m(\xi)$ where, as usual $\Phi_m\colon G\to M$ denotes the map $\Phi(\cdot,m)$. In what follows it will be more convenient to write $\xi_M=\tilde \xi$ for a fundamental vector f\/ield on $M$. We will also assume that $G$ is connected. In that case, a vector f\/ield $X$ on $M$ is $G$-invariant if and only if $[X,\tilde\xi]=0$, for all $\xi\in\la$.

There exist two ways to lift the action $\Phi$ to an action on $TM$. The f\/irst action is the $G$-action $\Phi^{TM}$ on $TM$, def\/ined by $\Phi^{TM}_g = T(\Phi_g)$ where again we write $\Phi_g=\Phi(g,\cdot)\colon M\to M$. With this action $\pi^{TM}\colon TM \to (TM)/G$, the so-called Atiyah bundle, is a principal $G$-bundle. Fundamental vector f\/ields of this action are vector f\/ields on $TM$ of the type ${\tilde\xi}^C$.

The tangent manifold $TG$ of a Lie group $G$ is also a Lie group. It may be identif\/ied with the semidirect product $G\times \la$, and its Lie algebra with $\la\times\la$. The second action on $TM$ is the $TG$-action, given by $T\Phi\colon TG \times TM \to TM$. In the current trivialization, this action may be written as
\begin{gather*}
(g,\xi)\cdot v_m=T\Phi_g \big(v_m+\tilde\xi(m)\big).
\end{gather*}
With this action $(TM)/(TG) =T(M/G)$ and $T\pi\colon TM \to T(M/G)$ is also a principal bundle, but with structure group $TG$ (see, e.g.,~\cite{KMS}). The fundamental vector f\/ields that correspond to~$T\Phi$ are linear combinations of the vector f\/ields on~$TM$ given by $(\xi,0)_{TM} = {\tilde\xi}^C$ and \mbox{$(0,\xi)_{TM} = {\tilde\xi}^V$}.

To clarify the concepts we introduce later on, we will often use local coordinates $(x^i)$ on $\bar M=M/G$ and coordinates $\big(x^A\big)=(x^i,x^a)$ on $M$ that are adapted to the bundle $\pi$. On the tangent manifold, however, we will use quasi-velocities with respect to a specif\/ic frame. If~$\{{E}_a\}$ is a basis for the (left) Lie algebra $\la$, we will denote the fundamental vector f\/ields (for the $G$-action on~$M$) by ${\tilde E}_a$. These vector f\/ields span the vertical distribution of $\pi$, and their Lie brackets are given by $[{\tilde E}_a,{\tilde E}_b] = - C_{ab}^c {\tilde E}_c$, if we denote the structure constants of $\la$ by $C^c_{ab}$.

We now assume that we have chosen a principal connection on $\pi$ (such a connection always exists). Let $X_i$ be the horizontal lifts, with respect to this connection, of the coordinate f\/ields~$\partial/\partial x^i$ on~$M/G$. The fact that the horizontal lifts are $G$-invariant vector f\/ields may be expressed as $[X_i,{\tilde E}_a] = 0$. Given that
\begin{gather*}
\big[X_i^C,{\tilde E}_a^{\{C,V\}}\big] = \big[X_i,{\tilde E}_a\big]^{\{C,V\}} = 0, \qquad \big[X_i^V,{\tilde E}_a^{\{C,V\}}\big]=0,
\end{gather*}
we see that both $X_i^C$ and $X_i^V$ are invariant vector f\/ields, for both the $G$-action on $TM$, and the $TG$-action on $TM$. The Lie bracket $[X_i,X_j] = R^a_{ij}{\tilde E}_a$ of these horizontal vector f\/ields represents the curvature of the connection.

We will denote the quasi-velocities with respect to the frame $\big\{X_i,{\tilde E}_a\big\}$ as $(v^i, v^a)$. Actually, since the vector f\/ields~$X_i$ project, we may conclude that the quasi-velocities $v^i$ can be identif\/ied with the natural f\/ibre coordinates on $M/G$, i.e., $v^i={\dot x^i}$. Some further immediate properties are included in the list below (see, e.g.,~\cite{CM2, MC} for more details):
\begin{alignat*}{5}
& X^C_i(v^j)=0,\qquad&& X^V_i(v^j)=\delta^j_i, \qquad&& X^C_i(v^a)=-R^a_{ij}v^j,\qquad&& X^V_i(v^a)=0,& \\
& {\tilde E}^C_a(v^i)=0,\qquad &&{\tilde E}^V_a(v^i)=0,\qquad && {\tilde E}^C_a(v^b)=C_{ac}^bv^c,\qquad && {\tilde E}^V_a(v^b)=\delta^b_a.&
\end{alignat*}
A SODE $\Gamma$ takes the form
\begin{gather*}
\Gamma = v^i X_i^C + v^a {\tilde E}^C_a + F^i X_i^V + F^a {\tilde E}^V_a
\end{gather*}
when expressed in this frame. It will be $G$-invariant when
\begin{gather*}
\big[\Gamma, {\tilde E}^C_a\big] = 0 \quad \Leftrightarrow\quad {\tilde E}_a^C(F^i)=0, \qquad {\tilde E}_a^C(F^b)= F^c C_{ac}^b.
\end{gather*}
It is important to realize, however, that a SODE can never be $TG$-invariant since, besides the previous properties, it would also have to satisfy $[\Gamma, {\tilde E}^V_a] = 0$, which is impossible because of the f\/irst term in
\begin{gather*}
\big[\Gamma, {\tilde E}^V_a\big] = - {\tilde E}_a^C - {\tilde E}_a^V(F^i)X_i^V - \big({\tilde E}_a^V\big(F^b\big) + v^cC^b_{ca}\big) {\tilde E}_b^V.
\end{gather*}
For this reason a SODE on $M$ can never be the horizontal lift of a principal connection on the $TG$-bundle $T\pi\colon TM \to T(M/G)$.

\section{Two lifted connections} \label{sectionlifted}

Assume that a principal connection on $\pi$ is given. There are many equivalent ways to represent this principal connection. For example, we may either consider it as being given by a $(1,1)$-tensor f\/ield $\omega$ on $M$ (its ``vertical projection operator'') or by a connection map $\varpi\colon TM \to \la$. These two approaches are related as $\omega(X)(m) = \widetilde{\varpi(X(m))}(m)$. A third way to def\/ine a connection makes use of its horizontal lift. The horizontal lift of a vector f\/ield $\bar X$ on $\bar M$ is the unique vector f\/ield~${\bar X}^h$ on $M$ that projects on $\bar X$ and that is such that $\omega\big({\bar X}^h\big) = 0$. In what follows we will use these three def\/initions simultaneously.

The principal connection lifts to a principal connection on each of the principal bundles $\pi^{TM}$ and $T\pi$. In \cite[Proposition~3]{MC}, it is shown that the complete lift $\omega^C$ (a $(1,1)$ tensor f\/ield on $TM$) represents a connection on the f\/ibre bundle $T\pi\colon TM \to T(M/G)$. Moreover, since ${\mathcal L}_{{\tilde \xi}}\omega =0$, the properties of complete lifts lead to ${\mathcal L}_{{\tilde \xi}^C}\omega^C=0$ and ${\mathcal L}_{{\tilde \xi}^V}\omega^C=0$, from which we may conclude that this connection is principal with respect to the structure group $TG$. The idea of using the complete lift of a connection on a bundle $Y \to X$ to def\/ine a connection on $TY \to TX$ comes from the paper \cite{Vilms} by Vilms, who used it in the context of vector bundles. For that reason, we will refer to $\omega^C$ as the Vilms connection.
 \begin{definition}
 The Vilms connection of a principal connection $\omega$ on $\pi\colon M \to \bar M$ is the $TG$-principal connection on $T\pi\colon TM \to T\bar M$ whose vertical projection operator is given by the tensor f\/ield $\omega^C$.
\end{definition}

The Vilms connection is characterized by its action on vector f\/ields on $TM$, and therefore by
\begin{gather*}
\omega^C\big(X^C_i\big)=0, \qquad \omega^C\big(X^V_i\big)=0, \qquad \omega^C\big({\tilde E}^C_a\big) = {\tilde E}^C_a, \qquad \omega^C\big({\tilde E}^V_a\big) = {\tilde E}^V_a.
\end{gather*}

\begin{proposition} \label{newprop} Let $\omega$ be a principal connection on $\pi$. The horizontal lift $H$ of the Vilms connection can be characterized in terms of the horizontal lift~$h$ of $\omega$, by means of the relations
\begin{gather*}
\big({\bar X}^C\big)^H = \big({\bar X}^h\big)^C \qquad \text{and} \qquad \big({\bar X}^V\big)^H = \big({\bar X}^h\big)^V,
\end{gather*}
for any vector field $\bar X$ on $\bar M$.
\end{proposition}
\begin{proof}
The horizontal lift of the Vilms connection maps a vector f\/ield $\bar Y$ on $T(M/G)$ onto a~vector f\/ield ${\bar Y}^H$ on~$TM$. Vector f\/ields on $T(M/G)$ are functional combinations of the coordinate vector f\/ields $\partial/\partial x^i$ and $\partial/\partial {\dot x}^i$. One may easily check that their horizontal lifts to vector f\/ields on~$TM$, for the Vilms connection, are
\begin{gather*}
\left(\fpd{}{x^i}\right)^H = X_i^C , \qquad \left(\fpd{}{{\dot x}^i}\right)^H= X_i^V.
\end{gather*}
From this, one can deduce the statement in the proposition.
\end{proof}

If we are given a SODE on $M/G$, $\bar\Gamma = {\dot x}^i\partial/\partial {x}^i+ f^i\partial/\partial {\dot x}^i$, its horizontal lift is the vector f\/ield on $TM$ given by
\begin{gather*}
{\bar\Gamma}^H = {\dot x}^i X_i^C+ f^i X_i^V.
\end{gather*}
Since ${\mathcal L}_{{\tilde E}_a^{\{C,V\}}}\omega^C=0$, we have $\big[{\tilde E}_a^{\{C,V\}}, {\bar\Gamma}^H \big]=0$, which expresses that ${\bar\Gamma}^H$ is $TG$-invariant. Its $TG$-reduced vector f\/ield is, of course, $\bar\Gamma$. However, ${\bar\Gamma}^H$ is {\em not} a SODE on~$M$: Essentially we are missing the term ``$v^a{\tilde E}^C_a$''. To understand the part we are missing, we need to invoke a second connection.

The second connection is the ``vertical lift'' of $\omega$. By this we do not mean the $(1,1)$ tensor f\/ield $\omega^V$ we had def\/ined above. We now need the interpretation of the connection on $\pi$ as the connection map $\varpi\colon TM \to \la$.
\begin{definition}
 The vertical connection of a principal connection $\varpi$ on $\pi\colon M \to \bar M$ is the $G$-principal connection on $\pi^{TM}\colon TM \to (TM)/G$ whose connection map $\varpi^V\colon TTM \to \la$ is given by $\varpi^V=\tau_M^*\varpi$, with $\tau_M\colon TM\to M$.
\end{definition}
For the corresponding interpretation of $\varpi^V$ as a $(1,1)$-form $\Omega$ on $TM$ one may calculate that
\begin{gather*}
\Omega\big(X^C_i\big)=0, \qquad \Omega\big(X^V_i\big) = 0, \qquad \Omega\big({\tilde E}^C_a\big) = {\tilde E}^C_a, \qquad \Omega\big({\tilde E}^V_a\big)=0.
\end{gather*}
It may easily be verif\/ied that this def\/ines a principal connection on the $G$-bundle $\pi^{TM}$: Among other properties, ${\mathcal L}_{{\tilde E}_a^C}\Omega = 0$. The tensor f\/ield $\Omega$ clearly dif\/fers from the tensor f\/ield $\omega^V$, since, for example, $ \Omega\big({\tilde E}^C_a\big) = {\tilde E}^C_a$, but $\omega^V\big({\tilde E}^C_a\big) = {\tilde E}^V_a$.

Assume that $\Gamma$ is a $G$-invariant SODE which reduces to the vector f\/ield $\hat\Gamma$ on $(TM)/G$. The vertical lift connection $\varpi^V$ can be used to reconstruct the integral curve $\gamma$ (through~$v_0$) of~$\Gamma$, starting from the integral curve $\hat \gamma$ (through $\pi^{TM}(v_0)$) of $\hat\Gamma$. This procedure is in fact valid for any principal bundle $N\to N/G$ with a given principal connection, not necessarily $N=TM$ (see, e.g., \cite[Proposition~1]{MC}, where it is explained in detail). Let ${\hat \gamma}^{\rm Hor}$ be the horizontal lifted curve (by means of the principal connection $\varpi^V$) of $\hat \gamma$ through~$v_0$. It is def\/ined by the properties that it remains horizontal everywhere, that it goes through $v_0$ and that it projects on $\hat \gamma$. Let~$\theta$ be the Maurer--Cartan form on~$G$, i.e., $\theta(v_g) = TL_{g^{-1}}v_g$ for $v_g\in T_gG$. The reconstruction theorem states that if one solves the equation $\theta(\dot g) = \varpi^V\big(\Gamma\circ {\hat \gamma}^{\rm Hor}\big)$ (with $g(0)=e$) for a curve~$g(t)$ in~$G$, then the relation between the integral curves is given by $\gamma(t) = \Phi^{TM}_{g(t)} {\hat \gamma}^{\rm Hor}(t)$. If we write the SODE $\Gamma$ as before as $\Gamma = v^i X_i^C + v^a {\tilde E}^C_a + F^i X_i^V + F^a {\tilde E}^V_a$, then the reconstruction equation is of the form $(\theta(\dot g))^a = v^a$, where $v^a$ are the vertical quasi-velocities of ${\hat \gamma}^{\rm Hor}$. It is clear that, in this respect, the reconstruction equation is essentially related to the term $v^a {\tilde E}^C_a$ of the SODE~$\Gamma$. In what follows, we will often make use of this vector f\/ield, and we will denote it by~$X_\omega$.

\begin{proposition}\label{pro:Xw}
 Let $\Gamma_0$ be an arbitrary SODE on $M$. The vector field $X_\omega = \Omega(\Gamma_0)$ on $TM$ is independent of the choice of~$\Gamma_0$. It is $T\pi$-vertical, $G$-invariant, but not $TG$-invariant.
\end{proposition}
 \begin{proof}
Since $\Gamma_0$ must take the form $v^a {\tilde E}^C_a + \cdots $, the vector f\/ield $\Omega(\Gamma_0) = v^a {\tilde E}^C_a$ is independent of the choice of the SODE $\Gamma_0$. It is clearly $T\pi$-vertical since $T\pi\circ\Omega=0$. It is also easy to see that this vector f\/ield is $G$-invariant, but not $TG$-invariant. Indeed, we have
\begin{gather*}
\big[X_\omega,{\tilde E}^C_b\big] = v^a \big[{\tilde E}^C_a,{\tilde E}^C_b\big] - {\tilde E}_b^C (v^a) {\tilde E}^C_a = 0,
 \end{gather*}
 but, on the other hand,
 \begin{gather*}
\big[X_\omega,{\tilde E}^V_b\big] = v^a \big[{\tilde E}^C_a,{\tilde E}^V_b\big] - {\tilde E}_b^V (v^a) {\tilde E}^C_a = -\big(v^a C^c_{ab}{\tilde E}^V_c + {\tilde E}^C_b\big).\tag*{\qed}
\end{gather*}\renewcommand{\qed}{}
\end{proof}

From its expression in quasi-velocities, it can be seen that a pointwise def\/inition for the vector f\/ield $X_\omega$ is $X_\omega(v) = \widetilde{(\varpi(v))}^C(v)$, for all $v\in TM$.

\section{Un-reducing second-order systems} \label{sectionsode}

We start with some general considerations for an arbitrary principal bundle $\mu\colon P\to P/G$.
\begin{definition}
A vector f\/ield $X\in\vectorfields{P}$ is an un-reduction of a vector f\/ield $\bar X \in\vectorfields{P/G}$ if all its integral curves project on those of $\bar X$.
\end{definition}
This is actually equivalent with saying that $X$ and $\bar X$ are $\mu$-related, since for any integral curve $c(t)$ of $X$ the property in the def\/inition means that $(\mu\circ c)(t)=(\bar c\circ\mu)(t)$, which, after dif\/ferentiating, becomes
\begin{gather*}
T\mu \circ X = \bar X \circ \mu.
\end{gather*}
For each principal connection $\varpi$ on $\mu$, the horizontal lift ${\bar X}^h$ is clearly an un-reduction. This shows that if we f\/ix a principal connection, any un-reduction $X$ is of the type $X=\bar X^h + W$, where $W$ may be any $\mu$-vertical vector f\/ield on $P$.

We now turn to the case of interest, where $\mu=T\pi$ and $P=TM$ with structure group $TG$. For any vector f\/ield $\bar X$ on $T(M/G)$, the Vilms connnection $\omega^C$ on $T\pi$ generates an un-reduc\-tion~${\bar X}^H$. We will be interested in the case where either $\bar X$ or $X$ is a SODE.

\begin{proposition} Let $\bar \Gamma$ be a vector field on $T(M/G)$, and let $\Gamma$ be an un-reduction of $\bar\Gamma$ on $TM$.
\begin{enumerate}\itemsep=0pt
\item[$(1)$] If $\Gamma$ is a SODE on $M$, then $\bar \Gamma$ is a SODE on $M/G$.
\item[$(2)$] If $\bar \Gamma$ is a SODE on $M/G$, then $\Gamma$ is a SODE on $M$ if, and only if, $\varpi^V \circ \Gamma=\varpi$.
\item[$(3)$] If $\bar \Gamma$ is a SODE on $M/G$, then $\Gamma$ is a SODE on $M$ if, and only if, $\Omega(\Gamma)=X_\omega$.
\end{enumerate}
\end{proposition}
\begin{proof}
(1) The f\/irst statement follows from the commutative diagrams below. If $\Gamma$ is an un-reduction and a SODE it satisf\/ies $TT\pi\circ \Gamma=\bar \Gamma$ and $T\tau_M\circ \Gamma=\Id_{TM}$. After applying $T\pi$ on the last relation, we get $T (\pi\circ \tau_M) \circ \Gamma =\Id_{T(M/G)}$ and, from the diagram, the left hand side becomes $T\tau_{M_G}\circ TT\pi\circ \Gamma= T\tau_{M_G}\circ \bar \Gamma$, from which it follows that $\bar \Gamma$ is a SODE.
$$
\includegraphics{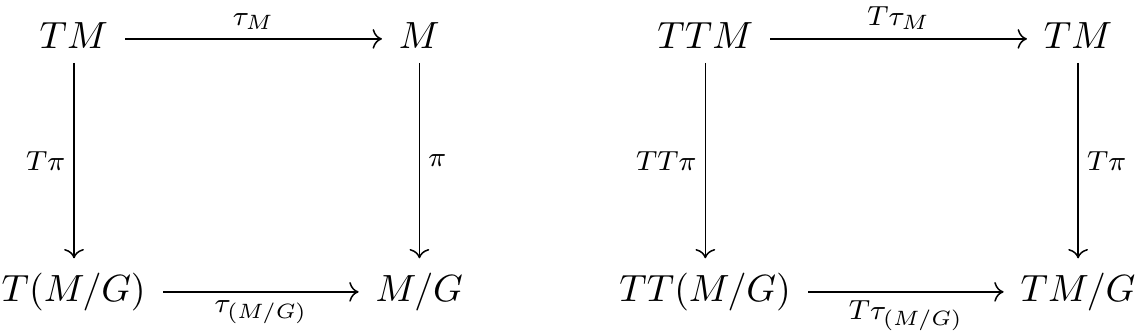}
$$

(2) The condition can be rewritten as $\varpi\circ T\tau_M \circ \Gamma=\varpi$. If $\Gamma$ is a SODE then $T\tau_M\circ \Gamma = \Id_{TM}$ and the statement follows. Assume now that $\varpi\circ T\tau_M \circ \Gamma=\varpi$. We may use the principal connection to decompose each $w_m \in T_mM$ into its horizontal and vertical part as $ (m,T\pi(w_m))^h+ (\varpi(w_m))_M(m)$. For $w_m = T\tau_M(\Gamma(v_m))$ we get, in view of the assumption that~$\Gamma$ is an un-reduction of a SODE $\bar \Gamma$ on $M/G$,
\begin{gather*}
w_m = (m,T(\pi\circ \tau_M)(\Gamma(v_m))))^h+ (\varpi(T\tau_M(\Gamma(v_m))))_M(m)\\
\hphantom{w_m}{} = (m,T(\tau_{M/G} \circ T\pi)(\Gamma(v_m))))^h + (\varpi(v_m))_M(m) \\
\hphantom{w_m}{} = (m,T(\tau_{M/G}) (\bar \Gamma(T\pi(v_m) )))^h + (\varpi(v_m))_M(m) = (m,T\pi(v_m))^h + (\varpi(v_m))_M(m).
\end{gather*}
From this it follows that $T\tau_M(\Gamma(v_m)) =w_m=v_m$, and that $\Gamma$ is a SODE on $M$.

(3) The relation between the connection map $\varpi$ of a connection, and its vertical projec\-tion~$\omega$ is given by $\omega(X)(m) = \widetilde{\varpi(X(m))}(m)$. In the case of the vertical connection this becomes
 \begin{gather*}
 \Omega(\Gamma)(v_m) = \big(\widetilde{\varpi^V(\Gamma(v_m))}\big)^C(v_m).
 \end{gather*}
If $\Gamma$ is an un-reduction of a SODE $\bar\Gamma$ then it will be SODE as well, in view of~(2), if and only if $\Omega(\Gamma)(v_m) = \big(\widetilde{\varpi(v_m)}\big)^C(v_m) = X_\omega(v_m)$.
\end{proof}

We may now introduce the following def\/inition.
 \begin{definition}
Let $\bar\Gamma$ be a SODE on $\bar M$ and $\omega$ be a principal connection on $\pi$. The primary un-reduced SODE of $(\bar\Gamma,\omega)$ is the SODE $\Gamma_1 = {\bar\Gamma}^H +X_\omega$ on $M$.
\end{definition}

In the frame $\big\{X_i,{\tilde E}_a\big\}$, $\Gamma_1$ takes the form $\Gamma_1= {\bar\Gamma}^H+ \Gamma_\omega = v^i X_i^C + f^i X_i^V + v^a {\tilde E}^C_a + 0 {\tilde E}^V_a $. The reason for calling $\Gamma_1$ an ``un-reduced'' SODE is given in the third observation below.
\begin{proposition}\label{pro:main}
Let $\bar\Gamma$ be a SODE on $\bar M$ and $\omega$ be a principal connection on $\pi$.
\begin{enumerate}\itemsep=0pt
\item[$(1)$] The primary un-reduced SODE $\Gamma_1$ of $(\bar\Gamma,\omega)$ is a $G$-invariant vector field on $TM$, but is not $TG$-invariant.
\item[$(2)$] The SODE $\Gamma_1$ is the unique un-reduced SODE on $M$ that satisfies $\omega^C(\Gamma) =X_\omega$.
\item[$(3)$] The base integral curves $c$ of $\Gamma_1$ project, via $\pi\colon M \to \bar M$, on base integral curves $\bar c$ of $\bar\Gamma$.
\item[$(4)$] $\Gamma_1$ is the unique SODE that projects on $\bar\Gamma$ and that has the property that it is tangent to the horizontal distribution of $\omega$.
\end{enumerate}
\end{proposition}
\begin{proof}
The f\/irst observation follows from the fact that ${\bar\Gamma}^H$ is $TG$-invariant, and therefore $\big[\Gamma_1, {\tilde \xi}^{\{C,V\}}\big] = \big[X_\omega, {\tilde \xi}^{\{C,V\}} \big]$. The second observation follows from the fact that the vector f\/ields~$X^C$ and~$Y^C$ are $T\rho$-related if and only if $X$ and $Y$ are $\rho$-related. A~similar observation holds for vertical lifts. One can easily prove this by considering the f\/lows of the involved vector f\/ields. In our case, we know that ${\bar\Gamma}^H$ and $\bar\Gamma$ are $T\pi$-related by def\/inition, and that ${\tilde E}_a$ is $\pi$-related to zero, whence~$X_\omega$ is $T\pi$-related to zero as well.

The observation that the SODE $\Gamma_1$ is $T\pi$-related to the SODE $\bar\Gamma$ means that its integral curves $\gamma$ in~$TM$ project, via $T\pi$, on integral curves $\bar\gamma$ of $\bar\Gamma$ in $T\bar M$. Since both vector f\/ields are SODEs, their integral curves are lifted curves. We conclude from this that the base integral curves $c$ of $\Gamma_1$ project, via $\pi\colon M \to \bar M$, on base integral curves~$\bar c$ of $\bar\Gamma$.

Horizontal vectors, expressed in quasi-velocities as $v_m = v^iX_i(m)+v^a{\tilde E}_a^C(m)$, have the property that $v^a = 0$. Let's write $\Gamma_1$ for now as $\Gamma_1 = v^i X_i^C + f^i X_i^V + v^a {\tilde E}^C_a + F^a {\tilde E}^V_a$. For the last property we need to show that $\Gamma_1(v^b) = 0$ is equivalent with $F^a=0$. This follows easily from the properties $X^C_i(v^b) = - R^b_{ij}v^j$, $X^V_i(v^a) =0$, ${\tilde E}^C_a(v^b) = C^b_{ac}v^c$ and ${\tilde E}^V_a(v^b) = \delta^b_a$.
\end{proof}

Integral curves of $\bar\Gamma$ satisfy $\ddot x^i = f^i(x,\dot x)$. For $\Gamma_1= {\bar\Gamma}^H+ X_\omega = v^i X_i^C + f^i X_i^V + v^a {\tilde E}^C_a + 0 {\tilde E}^V_a $, the integral curves satisfy, among other, $\ddot x^i = f^i(x,\dot x)$ and $\dot v^a=0$. The remaining equation is the ``reconstruction equation'' $(\theta(\dot g))^a=v^a$, where $\theta$ is the Maurer--Cartan form (see the paragraph above Proposition~\ref{pro:Xw}).

From the f\/irst statement in Proposition~\ref{pro:main}, the one which says that $\Gamma_1$ is $G$-invariant, we may conclude that $\Gamma_1$ can be reduced to a vector f\/ield on $(TM)/G$. This vector f\/ield is, obviously, not the same as the SODE $\bar\Gamma$, which is a vector f\/ield on $T(M/G)$. Vector f\/ields on $(TM)/G$ play, however, an essential role in so-called Lagrange--Poincar\'e reduction. The Euler--Lagrange equations of a $G$-invariant regular Lagrangian produce such a $G$-invariant SODE on $M$, and the integral curves of its reduced vector f\/ield on~$(TM)/G$ satisfy what are called the Lagrange--Poincar\'e equations. In this sense, one may think of the un-reduction of~\cite{unreduction}, which heavily relies on Lagrange--Poincar\'e reduction, as a process that compares the reduced vector f\/ield on $(TM)/G$ with the vector f\/ield $\bar\Gamma$ on $T(M/G)$.

Let $\bar c$ be a curve in $M/G$ which goes at $t=0$ through a point $\bar m$. Consider a point $m\in M$ with $\pi(m)=\bar m$. The horizontal lift of $\bar c$ at $m$ is the unique curve ${\bar c}^h_m$ which is def\/ined by the following three properties: (1) it projects on $\bar c$ for all $t$, (2) it satisf\/ies ${\bar c}^h_m(0)=m$, (3) it has only horizontal tangent vectors ${\dot{\bar c}}^h_m(t)$.

\begin{proposition} \label{propprim2}
The base integral curves of the primary un-reduced SODE $\Gamma_1$ of $(\bar\Gamma,\omega)$ through horizontal initial vectors are horizontal lifts of base integral curves of $\bar\Gamma$.
\end{proposition}

\begin{proof} Consider an initial value $\bar m$ in $M/G$ and an initial tangent vector $\bar v=v^i_0 \partial/\partial x^i|_{\bar m}$, and their corresponding base integral curve $\bar c_{\bar m,\bar v}(t)$ of $\bar\Gamma$. Let $m\in M$ be such that $\pi(m)=\bar m$, and consider the vector $v:=h(m,\bar v)$, the $\omega$-horizontal lift of $\bar v$ to $m$. In the frame $\{X_i, {\tilde E}_a\}$, we may write $v=v^i_0 X_i(m) +0 {\tilde E}_a(m)$. With the initial values $(m,v)$, the base integral curve $c_{m,v}(t)$ of $\Gamma$ satisf\/ies $\pi(c_{m,v}(t)) = \bar c_{\bar m,\bar v}(t)$ and also $v^a(t)=0$ (with this particular initial value). The solution of the remaining equation is then $\theta(\dot g)=0$, from which we see that $g(t)$ remains constant.

From the fact that $v^a(t)=0$, we conclude that the velocity ${\dot{\bar c}}^h_m(t)$ remains horizontal throughout. Given that also (1) and (2) are satisf\/ied, we may conclude that $c_{m,v}= ({\bar c}_{\bar m,\bar v})^h_m$. \end{proof}

If we had taken a dif\/ferent initial condition, say $v^a(0)=\xi^a$, we would get $v^a(t)=\xi^a$, and the corresponding curve $c$ would not be a horizontal lift.

From this proposition, we may derive an ``un-reduction algorithm''. Consider a SODE~$\bar\Gamma$ on~$\bar M$, a principal connection on $M\to \bar M$ and their corresponding ``lifted'' primary un-reduced SODE $\Gamma_1 =X_\omega +(\bar\Gamma)^H$. Calculate its integral curves~$c(t)$ in~$M$, but only for the very specif\/ic initial values we gave above: the initial velocity has to be horizontal. We have just shown that these curves are in fact the horizontal lifts of the curves $\bar c$ that we really want to know. If we project them down, we get the desired curves $\bar c$. Note that any f\/ield of initial velocities works as well: we could compute the integral curves of $\Gamma_1$ for \emph{any} set of initial conditions, and then project them in view of Proposition~\ref{pro:main}.

Other than the primary un-reduced SODE, there are many more SODEs that have the property that (some of) their integral curves project onto those of $\bar\Gamma$. All other un-reduction SODEs are of the form $\Gamma_2 =\Gamma_1 + V$, where $V$ is $T\tau$-vertical and $T\pi$-vertical. Since $\Gamma_1$, $X_i^V$ and~${\tilde E}_a^V$ are $T\pi$-related to $\bar\Gamma$, $\partial/\partial {\dot x}^i$ and $0$, respectively, this means that $V$ may be any vector f\/ield on~$TM$ of the type $V= V^a {\tilde E}_a^V$, which means that it is a vector f\/ield with $\Omega(V)=0$.
\begin{proposition}\label{prop0}
Let $\bar\Gamma$ be a SODE on $M/G$. Any vector field $\Gamma_2 = \Gamma_1 + V$, where $V$ is $T\pi$-vertical and such that $\Omega(V) = 0$, is a SODE which has the property that its base integral curves project on those of $\bar\Gamma$.
\end{proposition}

The full expression of $\Gamma_2$ is $\Gamma_2= v^i X_i^C + f^i X_i^V + v^a {\tilde E}^C_a + V^a {\tilde E}^V_a $. Its integral curves satisfy $\ddot x^i = f^i$ and $\dot v^a=V^a$.

If we only consider vector f\/ields $\Gamma_2$ that are $G$-invariant, then, since we know that $\Gamma_1$ is $G$-vertical, the coef\/f\/icients $V^a$ need to satisfy ${\tilde E}^C_a(V^b) = C^b_{ad}V^d$.

With the above, we have characterized all $G$-invariant SODEs $\Gamma_2$ which have the property that {\em all} their integral curves project to those of $\bar\Gamma$.
The systems that appear in \cite{unreduction} therefore belong to the SODEs we have discussed in this section, see Section~\ref{sectionLagrangian} for more details. In Section~\ref{sectionsecond}, we mention a dif\/ferent class of SODEs for which only a subclass of its integral curves project.

\section{A non-Lagrangian example}\label{sectioncanonical}

We will now consider an example where the dynamics is not of Lagrangian nature, and therefore falls out of the scope of \cite{unreduction}. This section intends to clarify some of the concepts we have introduced in the previous ones. First, we recall some well known facts. In general, if $\{Z_A\}$ is a frame on a manifold $M$, with quasi-velocities $v^A$, and if $\nabla$ is a linear af\/f\/ine connection with coef\/f\/icients $\nabla_{Z_A}{Z_B} = \gamma^D_{AB}Z_D$, then its quadratic spray is the SODE given by $v^AZ_A^C - \gamma^D_{AB}v^Av^B Z_D^V$. The base integral curves of this spray satisfy: ${\dot v}^D = - \gamma^D_{AB}v^Av^B$. This construction applies in particular to the case where the connection is the Levi-Civita connection of some (pseudo-)Riemannian metric on $M$.

Consider a connected Lie group $G$, with its Lie algebra $\la$ given by the Lie algebra of left-invariant vector f\/ields. The canonical af\/f\/ine connection on a Lie group $G$ can be def\/ined by its action on left-invariant vector f\/ields:
\begin{gather*}
\nabla_{\zeta_G^L}\eta_G^L =\frac{1}{2}[\zeta,\eta]^L_G,
\end{gather*}
where $\zeta_G^L$, $\eta_G^L$ stand for the left-invariant vector f\/ields of $\zeta,\eta\in\la$.

Among other properties, $\nabla$ has vanishing torsion; see~\cite{N} for a detailed discussion of this and other aspects of $\nabla$. We denote its quadratic spray by $\Gamma_G$. If the Lie-algebra $\la$ is (semi-simple or) compact then its Killing form def\/ines a (pseudo-)Riemannian metric on $G$ for which the canonical connection is its Levi-Civita connection. In other cases, although the canonical spray is a quadratic spray, the canonical connection may not be metrical and, moreover, its spray need not even be variational (see, e.g., \cite{Thompson} for some occurrences of this situation). It is also well-known that if the canonical connection is the Levi-Civita of some left-invariant Riemannian metric, then that metric is necessarily also right-invariant (see~\cite{ONeill}). The base integral curve of~$\Gamma_G$ through~$g$ and $\zeta\in\la$ (when $TG$ is left trivialized as $G\times \la$) is given by $t\mapsto g\exp(t\zeta)$. This is easy to see, as follows.

The connection coef\/f\/icients with respect to the frame $\{(E_A)_G^L\}$ of left-invariant vector f\/ields are $\gamma^D_{AB} = \frac12 C^D_{AB}$, which are skew-symmetric. The spray is then $\Gamma_G=v^A_L\big((E_A)^L_G\big)^C$ (with no $\big((E_A)^L_G\big)^V$ component) and the equation for its base integral curves is ${\dot v}_L^A=0$, and therefore $v^A_L(t)=\zeta^A$, for some constants~$\zeta^A$. Since $v^A_L$ are the quasi-velocities in the frame of left-invariant vector f\/ields, they are precisely the components of $\dot g$ left-translated to the Lie algebra, i.e., we have $g^{-1}\dot g=\zeta$. The solution of $\big(g^{-1}{\dot g}\big)^A =\zeta^A$ with $g(0)=e$ is precisely the one-parameter group of $\zeta$. The rest of the argument relies on the symmetry of~$\Gamma_G$.

Note that, from the expression $\Gamma_G=v^A_L(E_A)_L^C$, we may also write that $\Gamma_G(v_g)= \big(\zeta^L_G\big)^C(v_g)$, where $\zeta = TL_{g^{-1}}v_g = v^A E_A$.

Consider now a (closed) normal subgroup $N$ of $G$, and consider its right action on $G$. Then $K=G/N$ is again a Lie group, where multiplication is given by $(g_1N) (g_2N) = (g_1g_2) N$, where $gN$ denotes the left coset. We will write $\pi\colon G \to K=G/N$ for the projection and $\la$, $\lan$ and $\lak$ for the corresponding Lie algebras.

Since $K$ is a Lie group, it comes with its own canonical connection~$\nabla^K$. We wish to show that the sprays $\Gamma_K$ and $\Gamma_G$ of the two connections are related by means of an un-reduction process. Recall that, if that is the case, we only need to introduce coordinates on $G$ (and $N$) to be able to write down integral curves on~$K$.

To start the un-reduction process we need a principal $N$-connection on $\pi\colon G \to K$. Such a connection is, for example, available if we consider on $\la$ an $\operatorname{Ad}_N$-invariant inner product. (Remark that we do not assume that the inner product is $\operatorname{Ad}_G$-invariant. But, if that were the case, then it would generate a bi-invariant metric on $G$, whose Levi-Civita connection was $\nabla^G$.) The orthogonal complement $\lam$ of $\lan$ with respect to this inner product is then $\operatorname{Ad}_N$-invariant. In ef\/fect, this means that $N$ is reductive. We will write $\la = \lam \oplus \lan$. In these notations $T_e\pi(\la) = \lak$.
Consider the short exact sequence
\begin{gather*}
 0\to \lan\to\la\to\lak\to 0.
\end{gather*}
Its splitting $\lak \to \la$ with image $\lam$ will be denoted by $s$. From the property that $N$ is normal we get that $[\lan,\lan]\subset \lan$ and $[\lan,\lam]\subset\lan$ (i.e., $\lan$ is an ideal of $\la$). From the fact that $\lam$ is $\operatorname{Ad}_N$-invariant, we also get that $[\lan,\lam] \subset\lam$, and therefore $[\lan,\lam]=\{0\}$. Remark that we do not know much about $[\lam,\lam]$ (its vertical part is related to the curvature of the connection we will introduce next).

Let $P_\lam\colon \la\to\lam$ and $P_\lan\colon \la\to\lan$ denote the projections on $\lam$ and $\lan$, respectively. We may associate an $N$-principal connection on $\pi\colon G \to K$ to the decomposition $\la = \lam \oplus \lan$ by means of the connection map $\varpi(v_g) = P_\lan(TL_{g^{-1}}v_g)$. To see that this is a connection, recall that the inf\/initesimal generators of the $N$-action are given by the left-invariant vector f\/ields $\eta_G^L$ on $G$, associated to $\eta\in\lan$. Given that $\operatorname{Ad}_n \circ P_\lan = P_\lan \circ \operatorname{Ad}_n$ (where $n\in N$) we may easily verify that indeed $\varpi({\eta}_G^L(g)) = \eta$, for $\eta\in\lan$, and $\varpi(TR_n v_g) = \operatorname{Ad}_{n^{-1}}(\varpi(v_g))$.

Consider now, for $\xi\in \lak$, the vector f\/ield $\xi^L_K$ on $K$. We will show that its horizontal lift, with respect to the connection $\varpi$, is the left-invariant vector f\/ield $(s\xi)_G^L$. It is easy to see that this vector f\/ield is $N$-invariant, since for all inf\/initesimal generators of the action, $\big[\eta_G^L, (s\xi)_G^L\big] = [\eta, s\xi]_G^L =0$ for all $\eta\in\lan$. Moreover, $(s\xi)_G^L$ projects on $\xi^K_L$, since
\begin{gather*}
T\pi\big((s\xi)_G^L (g)\big) = T\pi (TL_g(s\xi) ) = TL_{gN}((\pi\circ s)(\xi)) = TL_{gN}(\xi) = \xi_L^K(\pi(g)).
\end{gather*}
Finally, $(s\xi)_G^L$ is horizontal, since $\varpi\big( (s\xi)_G^L(g)\big) = P_\lan(TL_{g^{-1}}(TL_g(s\xi))) = P_\lan(s\xi) =0$.

\begin{proposition}
Let $N$ be a normal subgroup of a Lie group $G$ and consider the principal connection $\varpi$ associated to an $\operatorname{Ad}_N$-invariant inner product on $\la$. Then, the canonical spray~$\Gamma_G$ on~$G$ is the primary un-reduced SODE of the canonical spray~$\Gamma_K$ on~$K=G/N$ and the connection~$\varpi$.
\end{proposition}

\begin{proof}
We f\/irst derive an expression for the Vilms-horizontal lift of the spray $\Gamma_K$. Consider a given $v_g\in T_gG$, and let $v_k=T\pi(v_g)\in T_{\pi(g)}K$. We will also use the notations $\zeta = TL_{g^{-1}}v_g$ and $\xi = TL_{k^{-1}}v_k$. By construction,
$\Gamma_K(v_k) = \big(\xi_L^K\big)^C(v_k)$. Since the complete lifts on $G$ and $K$ commute with the two connections (on the one hand the Vilms-horizontal lift, and on the other hand the horizontal lift, see Proposition~\ref{newprop}) we easily get that
\begin{gather*}
(\Gamma_K)^H(v_g) = \big(\big((\xi)^L_K\big)^h\big)^C(v_g) = \big((s\xi)_G^L\big)^C(v_g) = \big(\big( s \big(TL_{k^{-1}} T\pi (v_g)\big) \big)_G^L\big)^C(v_g)\\
\hphantom{(\Gamma_K)^H(v_g)}{} = \big(\big( s \big(T\big( \pi \circ L_{g^{-1}}\big) (v_g)\big) \big)_G^L\big)^C(v_g) = \big( \big((s \circ T \pi) \big(L_{g^{-1}} (v_g)\big) \big)_G^L\big)^C(v_g)\\
\hphantom{(\Gamma_K)^H(v_g)}{}= \big((P_{\lam} \zeta)_G^L\big)^C(v_g).
\end{gather*}

On the other hand, we have already stated that, in the current notations,
\begin{gather*}
X_\omega(v_g) = \big((\varpi(v_g))_G^L\big)^C(v_g) = \big(( P_{\lan} \zeta )_G^L\big)^C(v_g).
\end{gather*} Together, we get
\begin{gather*}
(\Gamma_K)^H(v_g) + X_\omega(v_g) = \big((P_{\lam} \zeta )_G^L\big)^C(v_g)+ \big(( P_{\lan} \zeta )_G^L\big)^C(v_g) = \big((\zeta )_G^L\big)^C(v_g),
\end{gather*}
which is exactly $\Gamma_G(v_g)$.
\end{proof}

{\bf Example: The general linear group as a bundle over $\R$.} In the set $M_{n\times n}(\R)$ of $n\times n$ real matrices we consider the Lie groups
\begin{gather*}
{\rm GL}^+(n)=\{A\in M_{n\times n}(\R)\colon \det(A)>0 \},\\
{\rm SL}(n)=\{A\in M_{n\times n}(\R)\colon \det(A)=1 \}.
\end{gather*}
The group ${\rm SL}(n)$ is a normal subgroup of~${\rm GL}(n)^+$. This follows from the fact that it is the kernel of the Lie group homomorphism $\det\colon {\rm GL}^+\to \R^+$, where $\R^+$ is the multiplicative group of positive real numbers. In particular, the determinant induces an isomorphism:
\begin{gather*}
\varphi\colon \ {\rm GL}^+(n)/{\rm SL}(n)\to \R^+.
\end{gather*}
The situation is summarized in the following diagram:
$$
\includegraphics{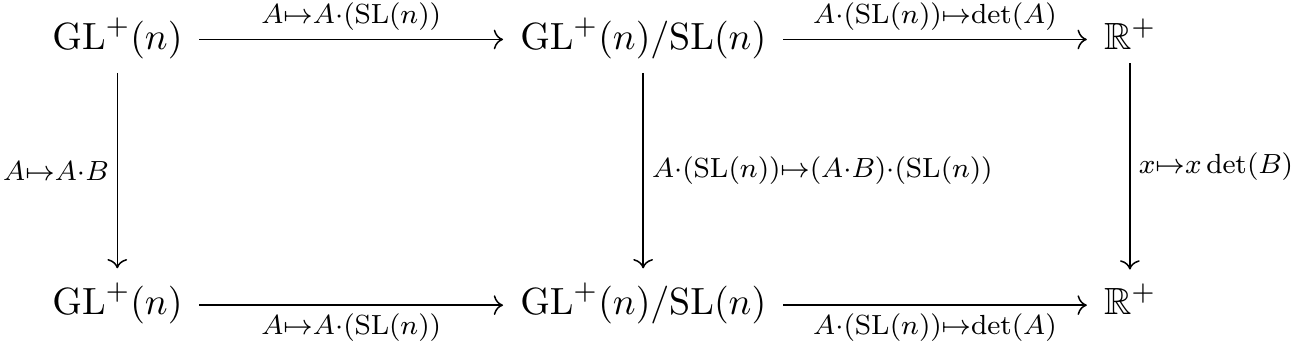}
$$
We can therefore think of ${\rm GL}^+(n)$ as a principal ${\rm SL}(n)$-bundle over $\R^+$, where ${\rm SL}(n)$ acts on the right on ${\rm GL}^+(n)$ and with the projection given by $\pi(A)=\det(A)$. We also observe that the tangent of the projection~$\pi$ at~$\Id$ (the $n\times n$ identity matrix) is $T_{\Id}\pi=\operatorname{trace}$.

The Lie algebra of ${\rm GL}^+(n)$ is $\gl(n)$ and the Lie algebra $\sli(n)$ of ${\rm SL}(n)$ consists of the real $n\times n$ traceless matrices. We consider the inner product on $\gl(n)$ given by $\langle A,B\rangle_{\gl(n)}=\operatorname{trace}\big(AB^{-1}\big)$. Since the trace is invariant under conjugation, $\langle \cdot,\cdot\rangle_{\gl(n)}$ def\/ines and $\operatorname{Ad}_{{\rm SL}(n)}$-invariant inner product on $\gl(n)$. Note that the identity matrix $\Id$ is orthogonal to $\sli(n)$, and therefore $\langle \cdot,\cdot\rangle_{\gl(n)}$ induces the splitting $\gl(n)=\sli(n)\oplus \langle \Id \rangle$. The associated map $s\colon \R\to \gl(n)$ is given by $s(\lambda)=\big(\frac{\lambda}{n}\big) \Id$. The situation is as follows:
$$
\includegraphics{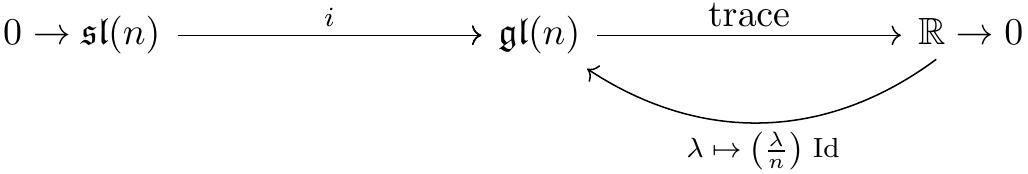}
$$
The horizontal space of the connection $\varpi$ is therefore given by left translation of the identity. A vector $v_A$ in $T_A({\rm GL}^+(n))$ is horizontal if it is of form $v_A=\mu A$, where $\mu\in\R$ (we identify the tangent space to a vector space with the vector space itself).
The horizontal lift of the tangent vector $\lambda \in T_x\R^+$ to $A\in\pi^{-1}(x)$ is given by the vector $v_A=\big(\frac{\lambda}{xn}\big) A$. Indeed, this vector is horizontal and it projects back, since $\pi\circ L_{A^{-1}} = L_{x^{-1}}\circ \pi$, and thus
\begin{gather*}
T\pi(v_A)= TL_x\circ T\pi\big(A^{-1}v_A\big)= x\operatorname{trace}\big(A^{-1}v_A\big)= \lambda.
\end{gather*}
We will use $x$ for the coordinate on $\R^+$. When written in terms of vector f\/ields, we may see from the above that, e.g., the horizontal lift of the coordinate vector f\/ield $\partial/\partial x$ on $\R^+$ is the vector f\/ield def\/ined by the map $X\colon {\rm GL}^+(n) \to {\rm GL}^+(n)\colon A \mapsto \frac{1}{n \det A}A$. We may also write $X=\frac{1}{nx} \id$, where $\id$ stands for the vector f\/ield given by the identity map on~${\rm GL}^+(n)$.

We will write $\bar\Gamma$ and $\Gamma$ for the canonical sprays on $\R^+$ and ${\rm GL}^+(n)$, respectively. The element $1\in\R$ is a basis for the Lie algebra of $\R^+$. Its left-invariant vector f\/ield on $\R^+$ is then $x \partial/\partial x$. If we use natural coordinates $(x,\dot x)\in T\R^+$, the quasi-velocity with respect to this vector f\/ield is then $w=\dot x/x$. Therefore, the canonical spray $\bar\Gamma$ on $\R^+$ is
\begin{gather*}
\bar\Gamma= w \left( x\fpd{}{x} \right)^C = \dot x \fpd{}{x} + \frac{{\dot x}^2}{x} \fpd{}{\dot x}.
\end{gather*}
Due to the properties in Proposition~\ref{newprop}, its Vilms horizontal lift is
\begin{gather*}
\bar\Gamma^H= \dot x\left(\left(\fpd{}{x}\right)^h\right)^C + \frac{{\dot x}^2}{x} \left(\left(\fpd{}{x}\right)^h\right)^V = \dot x X^C + \frac{{\dot x}^2}{x} X^V.
\end{gather*}
But, $X=\frac{1}{nx} \id$. Given that in general $(fY)^C = f Y^C + \dot f Y^V$, we may also write
\begin{gather*}
\bar\Gamma^H= \frac{\dot x}{nx} (\id)^C - \frac{{\dot x}^2}{nx^2} (\id)^V + \frac{{\dot x}^2}{nx^2} (\id)^V = \frac{\dot x}{nx} (\id)^C.
\end{gather*}

Consider now a basis $\{E_a\}\cup\Id$ of $\gl(n)$ ($a=1,\dots,n^2-1$). It is clear that the fundamental vector f\/ield $(\Id)_{{\rm GL}^+(n)}$ is simply the vector f\/ield $\id$ on ${\rm GL}^+(n)$. If $\{ v^a,v\}$ denote the quasi-velocities w.r.t.\ the frame $\big\{{\tilde E}_a,\id\big\}$ then we have that $v=\frac{\dot x}{nx}$. In the spray $\Gamma$,
\begin{gather*}
\Gamma =v^a \tilde{E}_a^C + v(\id)^C,
\end{gather*}
we recognize therefore in the f\/irst term the vector f\/ield $X_{\omega}$, and in the last term the Vilms horizontal lift $\bar\Gamma^H$.

\section{Un-reducing Lagrangian systems}\label{sectionLagrangian}

We now specialize to the case where the SODE $\bar\Gamma$ is the Lagrangian vector f\/ield of a regular Lagrangian $\ell$ (a smooth function on $T\bar M$). The vector f\/ield $\bar\Gamma$ on $T\bar M$ is then completely determined by two facts: (1) it is a SODE, (2) it satisf\/ies
\begin{gather*}
\bar\Gamma\left(\fpd{\ell}{\dot x^i}\right) - \fpd{\ell}{x^i}=0.
\end{gather*}
Written in coordinates, this characterizations says that if we write, in general, $\bar\Gamma = X^i \partial/\partial x^i + f^i \partial/\partial {\dot x}^i$ for a vector f\/ield on $T\bar M$, then, from (1) we know that $X^i ={\dot x}^i$, and from (2) we get that $f^i$ is determined by the Euler--Lagrange equations, when written in normal form. We will translate property (2) into one that the primary un-reduced SODE $\Gamma_1=X_\omega + {\bar\Gamma}^H$ satisf\/ies.

First, we recall the following observations for vector f\/ields on a principal $K$-bundle $p\colon Q\to Q/K$. When a vector f\/ield $W$ on $Q$ is $K$-invariant, the relation
$
\bar{W} \circ p=Tp \circ W
$ uniquely def\/ines its reduced vector f\/ield $\bar{W}$ on $Q/K$. Likewise, if $F\colon Q \to \mathbb{R}$ is a $K$-invariant function on $Q$ it can be reduced to a function $f\colon Q/K \to \mathbb{R}$ with
$
f \circ p\, = \, F$. The relation between these objects can easily be seen to be
\begin{equation*}
W(F) \, =\, W(f \circ p)\, = \, \bar{W}(f) \circ p,
\end{equation*}
which says that the function $\bar{W}(f)$ is the reduced function on $Q /K$ of the invariant function $W(F)$ on $Q$. We will use this for the case where $Q=TM$ and $K=TG$, so that $Q/K=T\bar M$ and $p=T\pi$.

Given $\ell$, we may def\/ine a function $L^H$ on $TM$ by $L^H(v) := \ell(T\pi(v))$. This function is not a regular Lagrangian on~$M$. But it is a $TG$-invariant function, which means that ${\tilde E}_a^C\big(L^H\big) = 0 = {\tilde E}_a^V\big(L^H\big)$. Its $TG$-reduced function is~$\ell$. The vector f\/ields $X_i^C$, $X_i^V$, $\bar\Gamma^H$ are all $TG$-invariant (since $\big[{\tilde E}_a^{\{C,V\}},X_i^C\big]=0$ and similar for~$X_i^V$ and $\bar\Gamma^H$), and their $TG$-reduced vector f\/ields on $\bar M$ are $\partial/\partial x^i$, $\partial/\partial {\dot x}^i$ and $\bar\Gamma$, respectively. From the relation above, we get that
\begin{gather*}
X_i^C\big(L^H\big) = \fpd{\ell}{{x}^i} \circ T\pi,\qquad
X_i^V\big(L^H\big) = \fpd{\ell}{{\dot x}^i} \circ T\pi, \qquad
\end{gather*}
and therefore, if $\bar\Gamma$ satisf\/ies the Euler--Lagrange equation, we obtain that ${\bar\Gamma}^H$ satisf\/ies:
\begin{gather*}
{\bar\Gamma}^H \big(X_i^V\big(L^H\big)\big) - X_i^C\big(L^H\big) =0.
\end{gather*}
We also know that $X_i^V\big(L^H\big)$ is a $G$-invariant function (this follows from $\big[{\tilde E}^C_a,X_i^V\big]\big(L^H\big)=0$), and therefore $X_\omega\big(X_i^V\big(L^H\big)\big) =0$. So, we may also write
\begin{gather*}
\Gamma_1\big(X_i^V\big(L^H\big)\big) - X_i^C\big(L^H\big) =0,
\end{gather*}
or, if ${\bar X}^h = {\bar X}^iX_i$ is an arbitrary $\omega$-horizontal lift of a vector f\/ield ${\bar X} = {\bar X}^i \partial/\partial x^i$ on~$\bar M$,
\begin{gather*}
\Gamma_1\big(\big({\bar X}^h\big)^V\big(L^H\big)\big) - \big({\bar X}^h\big)^C\big(L^H\big) =0,
\end{gather*}
or, equivalently,
\begin{gather*}
\Gamma_1\big(\big({\bar X}^V\big)^H\big(L^H\big)\big) - \big({\bar X}^C\big)^H\big(L^H\big) =0,
\end{gather*}
where $H$ now stands for the horizontal lift with respect to the Vilms connection~$\omega^C$.

The above equation completely determines the coef\/f\/icients $f^i$ in $\Gamma_1 = v^a{\tilde E}_a^C + v^i X_i^C + f^i X_i^V + 0 {\tilde E}_a^V$. We may thus conclude that the primary un-reduced SODE of a couple $(\ell,\omega)$ is the unique SODE that satisf\/ies any of the above equivalent expressions, and also $\omega^C(\Gamma_1)=X_\omega$ (This last property ensures that the coef\/f\/icient in~${\tilde E}_a^V$ is zero. See also Proposition~\ref{pro:main}).

Remark that if $\Gamma_1$ satisf\/ies this equation, then so does $\Gamma_2=\Gamma_1 +V$, where $V$ is any vector f\/ield of the type $V^a{\tilde E}_a^V$, since
 \begin{gather*}
 \Gamma_2\big(X_i^V\big(L^H\big)\big) = \Gamma_1\big(X_i^V\big(L^H\big)\big) + V^a {\tilde E}_a^V \big(X_i^V\big(L^H\big)\big) = \Gamma_1\big(X_i^V\big(L^H\big)\big) + V^a X_i^V \big({\tilde E}_a^V\big(L^H\big)\big) \\
 \hphantom{\Gamma_2\big(X_i^V\big(L^H\big)\big)}{}= \Gamma_1\big(X_i^V\big(L^H\big)\big) ,
 \end{gather*}
 because $\big[{\tilde E}_a,X_i\big] = 0$ and ${\tilde E}_a^V\big(L^H\big)=0$. In light of Proposition~\ref{prop0}, we conclude:
\begin{proposition}\label{prop1}
 Let $\ell$ be a regular Lagrangian on $\bar M$ and $\bar\Gamma$ its corresponding SODE. Let $\omega$ be a~principal connection on $\pi\colon M\to\bar M$. Any vector field $\Gamma_2$ on $M$ for which
\begin{itemize}\itemsep=0pt
\item[$(a)$] $\Gamma_2-\Gamma_1$ is $T\pi$-vertical,
\item[$(b)$] $\Omega(\Gamma_2-\Gamma_1)=0$,
\item[$(c)$] $\Gamma_2$ satisfies $\Gamma_2\big(\big({\bar X}^h\big)^V\big(L^H\big)\big) - \big({\bar X}^h\big)^C\big(L^H\big) =0$, for any vector field ${\bar X}$ on~$\bar M$,
\end{itemize} has the property that its base integral curves project on solutions of the Euler--Lagrange equations of~$\ell$.
\end{proposition}

Remark that the statement does not require the use of a specif\/ic vertical equation, as in~\cite{unreduction}. Each choice of $\Gamma_2$ leads to a~dif\/ferent vertical equation. When written in quasi-velocities, we get that $\Gamma_2$ is of the form $\Gamma_2= v^i X_i^C + v^a {\tilde E}^C_a+ f^i X_i^V +V^a {\tilde E}^V_a$. As noted above, the f\/irst two terms indicate that $\Gamma_2$ is a SODE, and the coef\/f\/icients $f^i$ in the third one are completely determined by the Euler--Lagrange equations of $\ell$. The only freedom left is therefore the choice of the coef\/f\/icients $V^a$ of the fourth term, which represents the choice of a vertical equation. In the next section we discuss an example where a specif\/ic choice for that freedom is naturally available. But the advantage of our approach is the same as the one that is claimed in \cite{unreduction}. All the conditions can be checked on the level of the manifold $M$, and we may do so in any coordinates on $M$, not necessarily those that are adapted to $M\to\bar M$.

\section{Curvature distortion}\label{sectionvertical}

Proposition~\ref{prop1} showed that there are many un-reductions $\Gamma_2$ of the same~$\bar\Gamma$. The best choice for~$\Gamma_2$ may depend on the specif\/ic example one considers. In some situations, there is a natural choice. This natural choice may lead to the introduction of what was called ``coupling distortion'' and ``curvature distortion'' in~\cite{unreduction}. We will concentrate here on the latter. To see its relation to the concepts we have introduced above, we may immediately restrict ourselves to the case of a~quadratic Lagrangian, without much loss of generality.

We consider again a principal f\/ibre bundle $M\to \bar M=M/G$, but now also a metric~$\bar g$ on~$\bar M$ and its corresponding geodesic spray ${\bar \Gamma}_{\bar g}$. The metric def\/ines a quadratic Lagrangian~$\ell$ on~$\bar M$. The goal of this section is to devise a method to obtain geodesics of~$\bar g$, without ever using coordinates on~$\bar M$. Integral curves of ${\bar\Gamma}_{\bar g}$ satisfy
\begin{gather*}
{\ddot x}^i + {\bar\Gamma}^i_{jk} {\dot x}^j{\dot x}^k = 0.
\end{gather*}

The idea of \cite{unreduction} is to use a reasonable construction of a metric $g$ on the un-reduced mani\-fold~$M$ and its geodesic spray~$\Gamma_g$. That spray, however, does not have the property that its geodesics project on those of $\bar g$. For that reason, we need to subtract some ``distortion'' terms from $\Gamma_g$ to get a SODE $\Gamma_2$ which does have that property. In what follows, we wish to obtain a relation between this $\Gamma_2$ and our primary un-reduced SODE~$\Gamma_1$.

The construction of the metric $g$ goes as follows. Let $B\colon \la \times \la\to \mathbb{R}$ be an $\operatorname{Ad}$-invariant symmetric and non-degenerate bilinear form on~$\la$ (i.e., a bi-invariant metric on~$G$). Consider again a principal connection on~$\pi$. Together with $\bar g$ we can form a (quadratic) Lagrangian on~$M$, $L=L^H+L^V$, where $L^H(v)=\ell(T\pi(v))=\frac{1}{2} g_{ij}v^iv^j$ and $L^V (v) = \frac{1}{2}B(\varpi(v),\varpi(v)) = \frac{1}{2}B_{ab}v^av^b $. The coef\/f\/icients $B_{ab}$ are constants that satisfy
\begin{gather*}
B_{ac}C^d_{ac} + B_{bd}C^d_{ac} =0.
\end{gather*}
This construction of a regular Lagrangian $L$ can also be found in \cite{CMR}, and our notations are chosen in such a way that they match with those in \cite{unreduction}.

Let us denote the (pseudo-)metric on $M$, associated to $L$, by $g$. By construction the mechanical connection of this metric is the connection we have started from, i.e., $g_{ai} =0$. The equations for the geodesics of $g$ are the so-called ``Wong equations'' (see again \cite{CMR}). An easy way to obtain these equations goes as follows. The geodesic spray $\Gamma_g$ is the Euler--Lagrange SODE of $L$. As such, it is the SODE determined by
\begin{gather*}
\Gamma_g\big(X^V(L)\big) -X^C(L) =0,
\end{gather*}
for any choice of vector f\/ield $X$ on $M$. If we use $X={\tilde E}_a$, then we know from the invariance of the Lagrangian that ${\tilde E}^C_a(L)=0$. Therefore the corresponding Euler--Lagrange equation is
 \begin{gather*}
 \Gamma_g\big(B_{ab}v^b\big) = 0,
 \end{gather*}
 meaning that the momentum $B_{ab}v^b$ is conserved along geodesics. This is, in essence, the ``vertical'' Wong equation. The ``horizontal'' Wong equation is the Euler--Lagrange equation we get by making use of $X=X_i$,
\begin{gather*}
 \Gamma_g\big(X_i^V(L)\big) -X_i^C(L) = 0.
\end{gather*}

Consider now again the SODE ${\bar\Gamma}_{\bar g}$ on $\bar M$ and its many un-reduced SODEs $\Gamma_2$, given in Proposition~\ref{prop1}. By construction, all these vector f\/ields satisfy
\begin{gather*}
 \Gamma_2\big(X_i^V\big(L^H\big)\big) -X_i^C\big(L^H\big) = 0.
\end{gather*}
Motivated by what we know about $\Gamma_g$, we now choose within that class, the particular SODE~$\Gamma_2$ which has the property that
\begin{gather*}
\Gamma_2(B_{ab}v^b) = 0.
\end{gather*}
It is easy to see that this last property uniquely determines a specif\/ic $\Gamma_2$ within the class of all SODEs that satisfy the conditions of Proposition~\ref{prop1}, since it f\/ixes the value of the coef\/f\/i\-cients~$V^a$ in $\Gamma_2= v^i X_i^C + v^a {\tilde E}^C_a+ f^i X_i^V +V^a {\tilde E}^V_a$. In this case, given that $B_{ab}$ is a non-degenerate constant matrix, we get simply
\begin{gather*}
\Gamma_2\big(v^b\big) = 0,
\end{gather*}
which is exactly the def\/ining relation of the primary un-reduced SODE $\Gamma_1$ within the class of all $\Gamma_2$'s, and thus $V^a=0$. Since $\Gamma_g$ satisf\/ies the same equation, it has the same coef\/f\/icients $V^a=0$, when written in quasi-velocities. The dif\/ference $A = \Gamma_g-\Gamma_1$ is therefore a vector f\/ield of the type $A=A^i X_i^V$. Given that $L=L^H + L^V$, and that $X_i^V\big(L^V\big) =0$, its coef\/f\/icients $A^i$ are completely determined by the dif\/ference between the respective horizontal equations, i.e., by the relation
\begin{gather*}
A \big(X_i^V\big(L^H\big)\big) -X_i^C\big(L^V\big) = 0.
\end{gather*}
If we derive the expressions of the curvature coef\/f\/icients from $[X_i,X_j] = R_{ij}^a {\tilde E}_a$, then $X^C_i(v^b) = - R^b_{ij}v^j$, we get that $X_i^V\big(L^H\big) = A^j{\bar g}_{ij}$ and $X_i^C\big(L^V\big)=-B_{ab}R^a_{ij}v^jv^b$ and thus is
\begin{gather*}A^k = - {\bar g}^{ik}B_{ab}R^a_{ij}v^bv^j.
\end{gather*}

The term $A$ in
\begin{gather*}
\Gamma_1 =\Gamma_g -A
\end{gather*}
is what is called ``curvature distortion'' in \cite{unreduction}. It is the term one needs to subtract from $\Gamma_g$ in order to get a SODE with the property that its base integral curves project on those of ${\bar\Gamma}_{\bar g}$. The (vertical) vector f\/ield $A$ may, of course, also be expressed in a coordinate-free manner, but we will not go into these details here.

As we saw, the ef\/fect of subtracting $A$ to $\Gamma_g$ is that it does not change the term of~$\Gamma_g$ in~${\tilde E}^V_a$, but that from the term in $X_i^V$ it cancels out the curvature term in the right-hand side of the horizontal Wong equation. The point is that both $\Gamma_g$ and $A$, and by the above construction also~$\Gamma_1$, can be computed in any coordinates on $M$. The above procedure is, in essence, the method that is applied in~\cite{unreduction}.

{\bf The f\/ibration of the rotation group over the sphere.} We will illustrate the discussion above by means of the realization of ${\rm SO}(3)$ as a~$\Sph^1$-bundle over the sphere $\Sph^2\subseteq \R^3$. More precisely, for the standard metric on the sphere, we will consider the geodesics of the metric determined by $L=L^H+L^V$ and explicitly compute the ``distortion'' term.

We identify each rotation about the origin of $\R^3$ with an orthogonal matrix. Such a rotation $R\in {\rm SO}(3)$ is determined by three consecutive counterclockwise rotations, def\/ined by the Euler angles $(\psi, \theta, \varphi)$ (following the convention of~\cite{ArnoldBook}):
\begin{gather*}R(\psi, \theta, \varphi) =
\left(\begin{matrix}
\cos \psi & -\sin \psi & 0\\
\sin \psi & \cos \psi & 0\\
 0 & 0 & 1
\end{matrix}\right)
\left(\begin{matrix}
1 & 0 & 0\\
0 & \cos \theta & -\sin \theta\\
0 & \sin \theta & \cos \theta
\end{matrix}\right)
\left(\begin{matrix}
\cos \varphi & -\sin \varphi & 0\\
\sin \varphi & \cos \varphi & 0\\
 0 & 0 & 1
\end{matrix}\right).
\\
\hphantom{R(\psi, \theta, \varphi)}{} =
\left(\!\begin{matrix}
\cos\psi\cos\varphi - \cos \theta\sin\psi\sin\varphi\!\! & -\cos\psi\sin\varphi - \cos \theta\sin\psi\cos\varphi\!\! & \sin\psi\sin \theta\\
\sin\psi\cos\varphi + \cos \theta \cos\psi\sin\varphi\!\! & -\sin\psi\sin\varphi + \cos \theta\cos\psi\cos\varphi\!\! & -\cos\psi\sin\theta\\
\sin\varphi\sin\theta & \cos\varphi\sin\theta & \cos\theta
\end{matrix}\!\right)\!.
\end{gather*}
The group $\Sph^1$ acts on ${\rm SO}(3)$ on the left by rotations about the $z$-axis, namely as:
\begin{gather*}
\big( \alpha\in \Sph^1, R(\psi,\theta,\varphi)\in {\rm SO}(3) \big)\mapsto
\left(\begin{matrix}
\cos \alpha & -\sin \alpha & 0\\
\sin \alpha & \cos \alpha & 0\\
 0 & 0 & 1
\end{matrix}\right)\cdot R(\psi, \theta, \varphi)\in {\rm SO}(3).
\end{gather*}
Note that two elements $R_1,R_2\in {\rm SO}(3)$ are in the same orbit if their last row, which parameterizes a sphere $\Sph^2\subset \R^3$, coincides. This def\/ines a principal $\Sph^1$-bundle $\pi\colon {\rm SO}(3)\to \Sph^2$. In terms of the Euler angles the projection is $\pi(\psi, \theta, \varphi)=(\theta, \varphi)$, and the inf\/initesimal generator (spanning the vertical distribution) is $\tilde E=\partial/\partial\psi$.

In what follows, we will use the principal connection on $\pi\colon {\rm SO}(3)\to \Sph^2$, given by the connection form $\varpi=d\psi+\cos\theta d\varphi$. This is, in fact, the mechanical connection of the invariant Lagrangian on $T({\rm SO}(3))$, given by
\begin{gather*}
\tilde L=\frac{1}{2}\big(\dot\theta^2+\dot\varphi^2\sin^2\theta\big)+\frac{1}{2}\big(\dot\psi+\dot\varphi\cos\theta\big)^2.
\end{gather*}
(For the mechanical connection, we regard $\tilde L$ as a metric and def\/ine the horizontal space of the connection as the space that is orthogonal to the vertical space of~$\pi$.) The Lagran\-gian~$\tilde L$ corresponds to a Lagrange top with equal moments of inertial, see~\cite{ArnoldBook}. For this principal connection, we may give the following basis of horizontal vector f\/ields on~${\rm SO}(3)$:
\begin{gather*}
X_1=\fpd{}{\theta},\qquad X_2=\fpd{}{\varphi}-\cos\theta\fpd{}{\psi}.
\end{gather*}
Quasi-velocities w.r.t.\ the frame $\big\{X_1,X_2,\tilde E\big\}$ will be denoted $\{v^1=\dot\theta, v^2=\dot\varphi,w = \dot\psi + \dot\varphi\cos\theta\}$.

Consider now the standard metric $\bar g=d\theta^2+\sin^2\theta d\phi^2$ on $\Sph^2$. Its geodesic equations are given by the Euler--Lagrange equations of the Lagrangian $\ell=\dot \theta^2+\sin^2 \dot\phi^2$ on $T\Sph^2$, or by the geodesic spray
\begin{gather*}
{\bar\Gamma}_{\bar g} = \dot\theta \fpd{}{\theta} + \dot\varphi \fpd{}{\varphi} +\sin\theta\cos\theta \dot\varphi^2 \fpd{}{\dot\theta} - 2\cot\theta\dot\theta\dot\varphi \fpd{}{\dot\varphi}.
\end{gather*}
Its primary un-reduced SODE is the vector f\/ield on $T({\rm SO}(3))$ given by
\begin{gather*}
\Gamma_1 = \Gamma_\omega + {\bar\Gamma}_{\bar g}^H = \dot\psi \fpd{}{\psi} + \dot\theta \fpd{}{\theta} + \dot\varphi \fpd{}{\varphi} + \sin\theta\cos\theta \dot\varphi^2 \fpd{}{\dot\theta}
- 2\cot\theta\dot\theta\dot\varphi \fpd{}{\dot\varphi} \\
\hphantom{\Gamma_1 = \Gamma_\omega + {\bar\Gamma}_{\bar g}^H =}{} + \big(\sin\theta\dot\theta \dot\varphi + 2\cot\theta\cos\theta \dot\theta\dot\varphi \big)\fpd{}{\dot\psi} .
\end{gather*}

We now contrast this with the procedure of \cite{unreduction}, where the system is un-reduced by subtracting curvature distortion from the Euler--Lagrange equations of a certain $L=L^H+L^V$. A bi-invariant metric on the Lie algebra of~$\Sph^1$ is specif\/ied by a real number~$B$. If we set $L^V=Bw^2$, we may write
\begin{gather*}
L=\dot \theta^2+\sin^2\theta \dot\varphi^2+B \big(\dot\psi + \dot\varphi\cos\theta\big)^2,
\end{gather*}
which determines a metric $g$ on ${\rm SO}(3)$. Its geodesic spray is
\begin{gather*}
\Gamma_g = \dot\psi \fpd{}{\psi} + \dot\theta \fpd{}{\theta} + \dot\varphi \fpd{}{\varphi} + \left(\sin\theta\cos\theta \dot\varphi^2 - B\sin\theta\dot\varphi w\right) \fpd{}{\dot\theta}\\
\hphantom{\Gamma_g =}{} +\left(- 2\cot\theta\dot\theta\dot\varphi + \frac{B}{\sin\theta} \dot\theta w \right)\fpd{}{\dot\varphi} + \big(\sin\theta\dot\theta \dot\varphi + 2\cot\theta\cos\theta \dot\theta\dot\varphi - B\cot\theta\dot\theta w\big)\fpd{}{\dot\psi},
\end{gather*}
where $w$ is short-hand for $\dot\psi + \dot\varphi\cos\theta$. Clearly, $\Gamma_g$ is not yet one of the un-reduced SODEs of~${\bar\Gamma}_{\bar g}$, and curvature distortion is required.

In the previous paragraph, we had concluded that $\Gamma_1$ is the unique un-reduced SODE within the class of vector f\/ields~$\Gamma_2$ which has the property that $\Gamma_1(Bw)=0$. From the expressions above, it follows that the curvature distortion is given by
\begin{gather*}
A=\Gamma_g-\Gamma_1 = -B\sin\theta\dot\phi w \fpd{}{\dot\theta} + \frac{B\dot\theta w}{\sin\theta} \left( \fpd{}{\dot\varphi} - \cos\theta\fpd{}{\dot\psi}\right).
\end{gather*}
This is clearly related to the coef\/f\/icient of the curvature $[X_\theta,X_\varphi]=\sin\theta\tilde E$, by means of the inverse of the metric matrix $({\bar g}_{ij}) = \left(\begin{matrix} 1 &0 \\ 0 &\sin^2\theta\end{matrix}\right)$.

To end this paragraph, we sketch an alternative approach, without invoking coordinates on~$\Sph^2$. We refer the reader to~\cite{ArnoldBook,EKMR} for details omitted here. The Lie algebra~$\so(3)$ of~${\rm SO}(3)$ is given by the set of skew symmetric matrices. It has the following basis
\begin{gather*}
e_1=\begin{pmatrix}
0 & 0 & 0\\
0 & 0 & -1\\
0 & 1 & 0
\end{pmatrix},\qquad
e_2=\begin{pmatrix}
0 & 0 & 1\\
0 & 0 & 0\\
-1 & 0 & 0
\end{pmatrix},\qquad
e_3=\begin{pmatrix}
0 & -1 & 0\\
1 & 0 & 0\\
0 & 0 & 0
\end{pmatrix}.
\end{gather*}
We will identify $\so(3)$ with the Lie algebra $\R^3$ (given by the cross product ``$\times$''):
\begin{gather*}
\R^3\ni\omega=
\begin{pmatrix}
\omega_1\\
\omega_2\\
\omega_3
\end{pmatrix}
\mapsto [\omega]=
\begin{pmatrix}
0 & -\omega_3 & \omega_2\\
\omega_3 & 0 & -\omega_1\\
-\omega_2 & \omega_1 & 0
\end{pmatrix}.
\end{gather*}
The left and right invariant vector f\/ields corresponding to the basis vectors will be denoted by~$(e_i)_\ell$ and~$(e_i)_r$, and we will use $(e^i)_\ell$ and $(e^i)_r$ for the dual basis. Then, on $\so(3)$, one may consider the invariant inner product $\langle A,B\rangle_{\so(3)}=\operatorname{trace}\big(AB^{-1}\big)$ and its associated metric on ${\rm SO}(3)$, for which $\{(e^i)_\ell\}$ and $\{(e^i)_r\}$ are orthonormal moving frames. In terms of the Euler angles, this metric is the one given by the Lagrangian $\tilde L$ we had mentioned before.

Suppose now that $\gamma\in\Sph^2$. If $\pi(R)=\gamma$, then we have $\dot\gamma=T\pi(\dot R)=-[\omega]\gamma=\gamma\times\omega$. The connection form is $\varpi=(e^3)_r$ and its curvature 2-form is the area element on $\Sph^2$. Finally, the horizontal lift of a vector $\dot\gamma$ to $R$ is the element $\dot R$ for which $R^{-1}\dot R=[\dot\gamma\times\gamma]$, see~\cite{EKMR}. One might also identify a rotation with an an oriented orthonormal frame $(v,w,v\times w)$ in $\R^3$ (the rotation needed to move the standard $\R^3$ basis onto the new frame). With this in mind, we pick $\hat e_1,\hat e_2\in\R^3$ such that $\{\hat e_1,\hat e_2,\gamma\}$ is orthonormal. The horizontal lift $v^h$ of a vector $v=(u_1\hat e_1+u_2\hat e_2)$ to an element $R\in {\rm SO}(3)$ is the tangent vector $\dot R$ such that
\begin{gather*}
R^{-1}\dot R=[(u_1\hat e_1+u_2\hat e_2)\times\gamma]=[u_2\hat e_1-u_1\hat e_2].
\end{gather*}
In other words, $v^h= u_2(\hat e_1)_r-u_1(\hat e_2)_r$. In this expression it is understood that $(\hat e_1)_r$, $(\hat e_2)_r$ are evaluated at the point $R$ to which we are lifting. Therefore, one has an explicit construction of horizontal vectors in ${\rm SO}(3)$, and one can compute $L=L^H+L^V$ without further dif\/f\/iculty.

\section{Extending the un-reduction method and further outlook} \label{sectionsecond}

So far we have considered SODEs $\Gamma_2$ which have the properties that {\em all} their base integral curves project on those of ${\bar\Gamma}$. We now indicate that there may also exist SODEs $\Gamma_3$ which have the property that only {\em some} of their base integral curves project on those of ${\bar\Gamma}$.

\looseness=-1 We consider a principal connection $\omega$ given. We will denote its horizontal distribution by ${\mathcal H} \subset TM$. Integral curves of $\bar\Gamma$ satisfy $\ddot x^i = f^i(x,\dot x)$. Consider again the primary un-reduced SODE $\Gamma_1 = v^i X_i^C + f^i(x^j,v^j) X_i^V + v^a {\tilde E}^C_a + 0 {\tilde E}^V_a $. Its integral curves satisfy $\ddot x^i = f^i(x,\dot x)$ and $\dot v^a=0$, together with some reconstruction equation that is primarily associated to the part $X_\omega=v^a {\tilde E}^C_a$ of $\Gamma_1$. Recall that we have shown in Proposition~\ref{propprim2} that integral curves of $\Gamma_1$ with a~horizontal initial velocity are horizontal lifts of base integral curves of $\bar\Gamma$ and that $\Gamma_1$ is tangent to ${\mathcal H}$.

Consider now a SODE of the type $\Gamma_3 = v^i X_i^C + F^i(x^j,v^j,v^a) X_i^V + v^a {\tilde E}^C_a + F^a(x^j,v^j,v^a) {\tilde E}^V_a$, but with the property that $F^i(x^j,v^j,v^a=0)=f^i(x^j,v^j)$ and $F^a(x^j,v^j,v^a=0) =0$, i.e., $\Gamma_3 |_{\mathcal H} = \Gamma_1 |_{\mathcal H}$. For example, besides the property on $F^i$, one could simply have that $F^a=0$. (This property, that there is no component along ${\tilde E}^V_a$, can be characterized by saying that the SODE $\Gamma_3$ satisf\/ies $\Omega(\Gamma_3) = \omega^C(\Gamma_3)$.) Integral curves of such a $\Gamma_3$ through $m_0$ with initial velocity $v^i_0 X_i(m_0) + v^a_0 {\tilde E}_a(m_0)$ satisfy, for sure, $v^a(t) = v^a_0$. As long as we consider integral curves of $\Gamma_3$ with a horizontal initial velocity, i.e., with $v^a_0 = 0$, they will satisfy $v^a(t) = 0$, and $(x^i(t),v^i(t))$ will be solutions of
 \begin{gather*}
 {\ddot x}^i = F^i(x^j,v^j,0)=f^i(x^j,v^j).
 \end{gather*}
With other words, some integral curves of $\Gamma_3$, namely the horizontal ones, will project on those of $\bar\Gamma$.

This situation actually also occurs in the example with Wong's equations. The quadratic spray $\Gamma_g$ is not one of the un-reduced SODEs of ${\bar\Gamma}_{\bar g}$, but it does have the property that $\Gamma_g |_{\mathcal H} = \Gamma_1 |_{\mathcal H}$, since $A|_{\{v^a=0\}}=0$. We can now relate this behaviour to what is called ``horizontal shooting'' in~\cite{MM, CH}.

We have shown that the geodesic spray $\Gamma_g$ satisf\/ies $\Gamma_g(B_{av} v^b)=0$. This means that its integral curves (i.e., the lifted curves in $TM$ of the base integral curves), whose f\/ibre coordinates when written in quasi-velocities are $(v^i(t),v^a(t))$, satisfy a conservation law (of momentum-type) $B_{av} v^b(t) = \mu_a$. Here $\mu_a$ are the components of an element $\mu\in\la^*$ along the basis $E_a$. The specif\/ic base integral curves which happen to have ``zero momentum'', $\mu_a=0$, are exactly those whose lifted curves remain horizontal for all $t$, $v^a(t) =0$. They, therefore, coincide with the horizontal lifts of base integral curves. Given that we know that $\Gamma_g |_{\mathcal H} = \Gamma_1 |_{\mathcal H}$, it follows that Proposition~\ref{propprim2} represents a generalization of the method of ``horizontal shooting'' for geodesic problems.

One may f\/ind in the literature some results on what could be called ``un-un-reduction''. The inverse procedure of un-reduction should of course be a kind of reduction process. It is, however, not Lagrange--Poincar\'e reduction, since the corresponding reduced Lagrange--Poincar\'e equations can not necessarily be associated to a SODE on $\bar M$. But, for an arbitrary $G$-invariant SODE $\Gamma$ on $TM$, it makes sense to wonder whether it is $T\pi$-related to a (yet to be determined) SODE~$\bar\Gamma$ on~$\bar M$. An answer to this question is given in the papers on submersive SODEs. A~SODE~$\Gamma$ on~$M$ is said to be ``submersive'' in \cite{KT,SPC} if there exists a projection $\pi\colon M \to N$ for which it is projectable to a SODE $\bar\Gamma$ on~$N$. In the theory of \cite{KT,SPC} the projection $\pi$ is part of the unknowns, but if we assume that $\pi$ is a given principal f\/ibre bundle $\pi\colon M \to\bar M$ from the outset, being submersive means that $\Gamma$ is an un-reduction of an (unknown) SODE $\bar \Gamma$ in~$\bar M$. The next proposition can be found as Theorem~3.1 in~\cite{KT}.
\begin{proposition} Let $M\to\bar M$ be a principal bundle and let $\Gamma$ be a SODE on $M$. Under the following two conditions~$\Gamma$ is submersive:
\begin{enumerate}\itemsep=0pt
\item[$(1)$] $\big[\Gamma, {\tilde \xi}^C\big] =0$ $($with other words, it is $G$-invariant$)$,
\item[$(2)$] $\big[\tilde\xi^V, \Gamma\big]-\tilde\xi^C$ is tangent to~${\la}^V$.
\end{enumerate}
\end{proposition}
\begin{proof}
If $\{E_a\}$ is a basis for $\la$, then ``being tangent to ${\la}^V$'' means that the expression in (2) is of the type $A^a {\tilde E}_a^V$.
If we write, as before, $\Gamma$ in the form
\begin{gather*}
\Gamma = v^i X_i^C + v^a {\tilde E}^C_a + F^i X_i^V + F^a {\tilde E}^V_a
\end{gather*}
with $F^i$, $F^a$ functions on $TM$, then condition (2) implies that ${\tilde E}^V_b(F^i)=0$, while from condi\-tion~(1) we get that ${\tilde E}^C_b(F^i)=0$ (and also that ${\tilde E}^C_b(F^a) = C^a_{bc}F^c$). Therefore, the function~$F^i$ is $TG$-invariant, and it def\/ines a function $f^i$ on $T\bar M$. The SODE on $\bar M$ given by $\bar\Gamma = {\dot x}^i\partial/\partial {x}^i+ f^i\partial/\partial {\dot x}^i$ is the one for which $\Gamma$ is an un-reduction.
\end{proof}

\looseness=-1 With the results of \cite{KT,SPC} in mind, it would be of interest to be able to ``un-reduce'' a~SODE~$\bar\Gamma$ on a manifold~$\bar M$ to any bundle $\pi$ for which $\bar M$ happens to be the base (given a connection on that bundle), not just a principal bundle $M\to \bar M =M/G$. Results which may point in that direction are Theorem~1.5 in~\cite{KT}, or Theorem~2.1 in~\cite{SPC}, which give conditions for a~SODE~$\Gamma$ on~$M$ to be submersive. Moreover, some of the concepts that we have mentioned in this paper transfer to this more general context. For example, one may still def\/ine a~Vilms connection on the bundle~$TM \to T\bar M$, corresponding to a connection on an arbitrary f\/ibre bundle~$M\to \bar M$. However, constructing something that is similar to the vertical connection~$\Omega$ (that has lead us to the def\/inition of~$X_\omega$ and the primary un-reduced SODE~$\Gamma_1$) seems to be a more challenging task.

\subsection*{Acknowledgements}

EGTA thanks the CONICET for f\/inancial support through a Postdoctoral Grant. TM is a~vi\-siting professor at Ghent University: he is grateful to the Department of Mathematics for its hospitality.

\pdfbookmark[1]{References}{ref}
\LastPageEnding

\end{document}